\documentclass{article}
\usepackage{fancyhdr}
\usepackage{setspace, amsmath, amsthm, amssymb, amsfonts, amscd, epic, graphicx, ulem, dsfont}
\usepackage{amsthm}
\usepackage{arxiv}

\usepackage[T1]{fontenc}
\usepackage{answers,mathptmx}
\newtheorem{theorem}{Theorem}

\newtheorem{lemma}[theorem]{Lemma}
\newtheorem{proposition}[theorem]{Proposition}
\newtheorem{corollary}[theorem]{Corollary}

\newtheorem{remark}[theorem]{Remark}

\title{On the nonexistence of harmonic and bi-harmonic maps}

\author{
  Ahmed Mohammed Cherif\\
Mascara University, Faculty of Exact Sciences, \\Laboratory of Geometry, Analysis, Controle and Applications, 29000, Algeria.\\
  \texttt{a.mohammedcherif@univ-mascara.dz}
}

\begin{document}
\maketitle

\begin{abstract}
In this paper, we study the existence of
harmonic and bi-harmonic maps into
Riemannian manifolds admitting a conformal vector field, or
a nontrivial Ricci solitons.
\end{abstract}

\keywords{Harmonic maps; Bi-harmonic maps; Ricci solitons; Conformal vector fields.}

\maketitle

\section{Preliminaries and Notations}
We give some definitions.
(1) Let $(M,g)$ be a  Riemannian manifold. By $R$ and $\operatorname{Ric}$
we denote respectively the Riemannian curvature tensor and the Ricci tensor of  $(M,g)$.
Thus $R$ and $\operatorname{Ric}$ are defined by:
\begin{equation}\label{eq1.1}
    R(X,Y)Z=\nabla_X \nabla_Y Z-\nabla_Y \nabla_X Z-\nabla_{[X,Y]}Z,
\end{equation}
\begin{equation}\label{eq1.2}
    \operatorname{Ric}(X,Y)=g(R(X,e_i)e_i,Y),
\end{equation}
where $\nabla$ is the Levi-Civita connection with respect to $g$, $\{e_i\}$ is an orthonormal frame,  and $X,Y,Z\in\Gamma(TM)$.
The divergence of $(0,p)$-tensor $\alpha$ on $M$ is defined by:
\begin{equation}\label{1.2}
    (\operatorname{div} \alpha)(X_1,...,X_{p-1})=(\nabla_{e_i}\alpha)(e_i,X_1,...,X_{p-1}),
\end{equation}
where $X_1,...,X_{p-1}\in\Gamma(TM)$, and $\{e_i\}$ is an orthonormal frame. Given a smooth function $\lambda$ on $M$,
the gradient of $\lambda$ is defined by:
\begin{equation}\label{eq1.3}
    g(\operatorname{grad}\lambda,X)=X(\lambda),
\end{equation}
the Hessian of $\lambda$ is defined by:
\begin{equation}\label{eq1.5}
    (\operatorname{Hess} \lambda)(X,Y)=g(\nabla_X \operatorname{grad}\lambda,Y),
\end{equation}
where $X,Y\in\Gamma(TM)$ (for more details, see for example \cite{ON}).\\
(2) A vector field $\xi$ on a Riemannian manifold $(M,g)$ is called a conformal if $\mathcal{L}_{\xi}g=2fg$, for some smooth function $f$ on $M$, where
$\mathcal{L}_{\xi}g$ is the Lie derivative of the metric $g$ with respect to $\xi$, that is:
\begin{equation}\label{eq1.6}
    g(\nabla_X \xi,Y)+g(\nabla_Y\xi,X)=2fg(X,Y), \quad X,Y\in\Gamma(TM).
\end{equation}
The function $f$ is then called the potential function of the conformal vector field $\xi$.
If $\xi$ is conformal with constant potential function $f$,
then it is called  homothetic, while $f=0$ it is Killing (see \cite{BW}, \cite{WH}, \cite{yano}).\\
(3) A Ricci soliton structure on a Riemannian manifold $(M,g)$ is the choice of a
smooth vector field $\xi$ satisfying the soliton equation:
\begin{equation}\label{eq1.6}
   \operatorname{Ric}+\frac{1}{2}\mathcal{L}_\xi g=\lambda g,
\end{equation}
for some constant $\lambda\in\mathbb{R}$,  where $\mathcal{L}_{\xi}g$ is the Lie derivative of the metric $g$ with respect to $\xi$.
The Ricci soliton $(M,g,\xi,\lambda)$ is said to be shrinking,
steady or expansive according to whether the coefficient $\lambda$ appearing in equation (\ref{eq1.6}) satisfies
$\lambda > 0$, $\lambda = 0$ or $\lambda < 0$.
In the special case where $\xi =\operatorname{grad}f$, for some smooth function $f$ on $M$, we say
that $(M,g,\operatorname{grad}f,\lambda)$ is a gradient Ricci soliton with potential $f$. In this situation, the soliton equation reads:
\begin{equation}\label{eq1.7}
   \operatorname{Ric}+\operatorname{Hess}f=\lambda g,
\end{equation}
(see \cite{H1}, \cite{H2}, \cite{SMA}). If $\xi=0$, we recover the definition of an Einstein metric with Einstein constant $\lambda$. If $(M, g)$ is not Einstein, we call the soliton nontrivial.\\
(4) A  vector field $\xi$ on a Riemannian manifold $(M, g)$ is said to
be a Jacobi-type vector field if it satisfies:
\begin{equation}\label{eq1.8}
    \nabla_X\nabla_X \xi-\nabla_{\nabla_X X}\xi+R(\xi,X)X=0,\quad X\in\Gamma(TM).
\end{equation}
Note that, there are Jacobi-type vector fields on a Riemannian manifold
which are not Killing vector fields (see \cite{Deshmukh}).\\
(5) Let $\varphi:(M,g)\rightarrow(N,h)$ be a smooth map between two Riemannian manifolds,
$\tau(\varphi)$ the tension field of $\varphi$ given by:
\begin{equation}\label{eq1.9}
\tau(\varphi)=\operatorname{trace}\nabla d\varphi=\nabla^{\varphi}_{e_i}d\varphi(e_i)-d\varphi(\nabla^{M}_{e_i}e_i),
\end{equation}
where $\nabla^{M}$ is the Levi-Civita connection of $(M,g)$, $\nabla^{\varphi}$ denote the pull-back connection on $\varphi^{-1}TN$
and $\{e_i\}$ is an orthonormal frame on $(M,g)$.
Then $\varphi$ is called harmonic if the tension field vanishes, i.e. $\tau(\varphi)=0$ (see \cite{BW}, \cite{CMO}, \cite{ES}, \cite{YX}).
We define the index form for harmonic maps
by (see \cite{C}, \cite{OND}):
\begin{equation}\label{eq1.10}
I(v,w)=\int_{M}h(J_{\varphi}(v),w)v^g,\quad v,w\in\Gamma(\varphi^{-1}TN)
\end{equation}
(or over any compact subset $D\subset M$), where:
\begin{eqnarray}\label{eq1.11}
J_{\varphi}(v)
   &=&\nonumber-\operatorname{trace}R^N(v,d\varphi)d\varphi-\operatorname{trace}(\nabla^{\varphi})^2 v  \\
   &=&  -R^N(v,d\varphi(e_i))d\varphi(e_i)-\nabla^{\varphi}_{e_i}\nabla^{\varphi}_{e_i}v+\nabla^{\varphi}_{\nabla^{M}_{e_i}e_i}v,
\end{eqnarray}
$R^N$ is the curvature tensor of $(N,h)$, $\nabla^{N}$ is the Levi-Civita connection of $(N,h)$,
and $v^g$ is the volume form of $(M,g)$
(see \cite{BW}). If $\tau_2(\varphi)\equiv J_{\varphi}(\tau(\varphi))$ is null on $M$,
then $\varphi$ is called a bi-harmonic map (see \cite{CMO}, \cite{Jiang}, \cite{LO}).

\section{Main Results}

\subsection{Harmonic maps and conformal vector fields}
\begin{proposition}\label{proposition3.1}
Let $(M,g)$ be a compact orientable  Riemannian manifold without boundary,
and $(N,h)$  a Riemannian manifold admitting a conformal vector field $\xi$ with potential function $f>0$ at any point.
Then, any harmonic map $\varphi$ from $(M,g)$ to $(N,h)$ is constant.
\end{proposition}
\begin{proof}
Let $X\in\Gamma(TM)$, we set:
\begin{equation}\label{eq3.1}
    \omega(X)=h\big(\xi\circ\varphi,d\varphi(X)\big),
\end{equation}
let $\{e_i\}$ be a normal orthonormal frame at $x\in M$, we have:
\begin{equation}\label{eq3.2}
    \operatorname{div}^M\omega
   = e_i\big[h\big(\xi\circ\varphi,d\varphi(e_i)\big)\big],
\end{equation}
by equation (\ref{eq3.2}), and the harmonicity condition of $\varphi$, we get:
\begin{eqnarray}\label{eq3.3}
\operatorname{div}^M\omega
   &=& h\big(\nabla^\varphi _{e_i} (\xi\circ\varphi),d\varphi(e_i)\big),
\end{eqnarray}
since $\xi$ is a conformal vector field,
we find that:
\begin{equation}\label{eq3.4}
  \operatorname{div}^M\omega
  =(f\circ\varphi)  h\big(d\varphi(e_i) ,d\varphi(e_i)\big)=(f\circ\varphi)|d\varphi|^2,
\end{equation}
the Proposition \ref{proposition3.1} follows from equation (\ref{eq3.4}), and the divergence theorem (see \cite{BW}), with $f>0$  on $N$.
\end{proof}

\begin{remark}
(1) Proposition \ref{proposition3.1} remains true if the potential function $f<0$ on $N$ (consider the conformal vector field $\bar{\xi}=-\xi$).\\
(2) If the potential function is non-zero constant, that is $\mathcal{L}_\xi h=2kh$ on $(N,h)$ with $k\neq0$,
 then any harmonic map $\varphi$ from a compact orientable  Riemannian manifold without boundary  $(M,g)$ to $(N,h)$  is necessarily constant (see \cite{cherif}).\\
(3) An harmonic map from
  a compact orientable  Riemannian manifold without boundary to
 a Riemannian manifold admitting a Killing vector field  is not necessarily constant (for example the identity map on the unit $(2n+1)$-dimensional sphere on $\mathbb{R}^{2n+2}$, note that
the unit odd-dimensional sphere admits a Killing vector field (see \cite{BL1}).
\end{remark}
From Proposition \ref{proposition3.1} we get the following result:
\begin{corollary}
Let $(\overline{N},\overline{h})$ be an $n$-dimensional Riemannian manifold which admits a Killing vector field $\overline{\xi}$. Consider $(N,h)$ a  Riemannian hypersurface of $(\overline{N},\overline{h})$ such that $h$ is the induced metric of $\overline{h}$ on $N$. Suppose that:
\begin{itemize}
  \item $(N,h)$ is totally umbilical, that is:
$$B(X,Y) =\rho h(X,Y)\eta,\quad\forall X,Y\in\Gamma(TN),$$ for some smooth function $\rho$ on $N$, where $B$ is the second fundamental form of $N$ on $\overline{N}$ given by $B(X,Y)=(\overline{\nabla}_X Y)^\perp$,  $\overline{\nabla}$ is the Levi-Civita connection on $\overline{N}$, and $\eta$ is the unit normal to $N$;
  \item the function $\overline{h}(\overline{\xi},H)\neq0$  everywhere on  $N$, where
$H$ is the mean curvature of $(N,h)$ given by the formula:
$$H=\frac{1}{n-1}\operatorname{trace}_h B.$$
\end{itemize}
Then, any harmonic map from
a compact orientable  Riemannian manifold without boundary to $(N,h)$ is constant.
\end{corollary}

\begin{proof}
It is possible to express $\overline{\xi}$ as $\overline{\xi}=\xi+f\eta$, where $\xi$ is tangent to $N$ and $f$ is a smooth function on $N$. Thus we have:
\begin{equation}\label{eq5.1}
    (\mathcal{L}_{\overline{\xi}} \overline{h})(X,Y)=(\mathcal{L}_\xi h)(X,Y)+f\{\overline{h}(\overline{\nabla}_X \eta,Y)+\overline{h}(\overline{\nabla}_Y \eta,X)\},
\end{equation}
where $X,Y\in\Gamma(TN)$ (see \cite{Duggal}),
by equation (\ref{eq5.1}) with  $ \mathcal{L}_{\overline{\xi}} \overline{h}=0$, we get:
\begin{equation}\label{eq5.2}
   (\mathcal{L}_\xi h)(X,Y)=2f\overline{h}(\eta,B(X,Y)),
\end{equation}
since  $N$ is totally umbilical, (\ref{eq5.2}) becomes:
\begin{equation}\label{eq5.3}
   (\mathcal{L}_\xi h)(X,Y)=2f\rho h(X,Y),
\end{equation}
the Corollary follows from Proposition \ref{proposition3.1} and equation (\ref{eq5.3}) with:
 $$f\rho=\overline{h}(\overline{\xi},\eta)\overline{h}(H,\eta)=\overline{h}(\overline{\xi},H).$$
\end{proof}
In the case of non-compact Riemannian manifold, we obtain the following results:
\begin{theorem}\label{theorem3}
Let $(M,g)$ be a complete non-compact Riemannian manifold,
and $(N,h)$  a Riemannian manifold admitting a conformal vector field $\xi$ with potential function $f>0$ at any point.
If $\varphi:(M,g)\longrightarrow(N,h)$ is harmonic map, satisfying:
\begin{equation}\label{eq3.5}
    \int_M \frac{|\xi\circ\varphi|^2}{f\circ\varphi} v^g<\infty,
\end{equation}
then $\varphi$ is constant.
\end{theorem}

\begin{proof}
Let $\rho$ be a smooth function with compact support on $M$, we set:
\begin{equation}\label{eq3.6}
    \omega(X)=h\big(\xi\circ\varphi,\rho^2d\varphi(X)\big),\quad X\in\Gamma(TM).
\end{equation}
Let $\{e_i\}$ be a normal orthonormal frame at $x\in M$, we have:
\begin{equation}\label{eq3.7}
    \operatorname{div}^M\omega
   = e_i\big[h\big(\xi\circ\varphi,\rho^2d\varphi(e_i)\big)\big],
\end{equation}
by equation (\ref{eq3.7}), and the harmonicity condition of $\varphi$, we get:
\begin{eqnarray}\label{eq3.8}
\operatorname{div}^M\omega
   &=& \nonumber h\big(\nabla^\varphi _{e_i} (\xi\circ\varphi),\rho^2d\varphi(e_i)\big)
    +h\big( \xi\circ\varphi,\nabla^\varphi _{e_i}\rho^2d\varphi(e_i)\big)\\
    &=&\rho^2h\big(\nabla^\varphi _{e_i} (\xi\circ\varphi),d\varphi(e_i)\big)
    +2\rho e_i(\rho)h\big(\xi\circ\varphi,d\varphi(e_i)\big),
\end{eqnarray}
since $\xi$ is a conformal vector field with potential function $f$,
we find that:
\begin{equation}\label{eq3.9}
  \rho^2h\big(\nabla^\varphi _{e_i} (\xi\circ\varphi),d\varphi(e_i)\big)
  =(f\circ\varphi)\rho^2  h\big(d\varphi(e_i) ,d\varphi(e_i)\big),
\end{equation}
by  Young's inequality we have:
\begin{equation}\label{eq3.10}
    -2\rho e_i(\rho)h\big(\xi\circ\varphi,d\varphi(e_i)\big)
    \leq \lambda \rho^2|d\varphi|^2+\frac{1}{\lambda}e_i(\rho)^2|\xi\circ\varphi|^2,
\end{equation}
for all function $\lambda>0$ on $M$, because of the inequality: $$|\sqrt{\lambda}\rho d\varphi(e_i)+\frac{1}{\sqrt{\lambda}}e_i(\rho)(\xi\circ\varphi)|^2\geq0.$$
From (\ref{eq3.8}), (\ref{eq3.9}) and (\ref{eq3.10}) we deduce the inequality:
\begin{equation}\label{eq3.11}
    (f\circ\varphi)\rho^2|d\varphi|^2-\operatorname{div}^M\omega
    \leq \lambda \rho^2|d\varphi|^2+\frac{1}{\lambda}e_i(\rho)^2|\xi\circ\varphi|^2,
\end{equation}
let $\lambda=(f\circ\varphi)/2$, by (\ref{eq3.11}) we have:
\begin{equation}\label{eq3.12}
    \frac{1}{2}(f\circ\varphi)\rho^2|d\varphi|^2-\operatorname{div}^M\omega
    \leq \frac{2}{f\circ\varphi}e_i(\rho)^2|\xi\circ\varphi|^2,
\end{equation}
by the divergence theorem, and (\ref{eq3.12}) we have:
\begin{equation}\label{eq3.13}
   \frac{1}{2} \int_M (f\circ\varphi)\rho^2|d\varphi|^2 v^g
    \leq 2\int_M e_i(\rho)^2\frac{|\xi\circ\varphi|^2}{f\circ\varphi} v^g.
\end{equation}
Consider the smooth function $\rho=\rho_R$ such that, $\rho\leq1$ on $M$, $\rho=1$ on the ball $B(p,R)$, $\rho=0$ on $M\backslash B(p,2R)$
and $|\operatorname{grad}^M \rho|\leq\frac{2}{R}$ (see \cite{pa}). From (\ref{eq3.13}) we get:
\begin{equation}\label{eq3.14}
    \frac{1}{2}\int_M(f\circ\varphi)\rho^2|d\varphi|^2 v^g
    \leq \frac{8}{R^2}\int_M\frac{|\xi\circ\varphi|^2}{f\circ\varphi} v^g,
\end{equation}
since $ \int_M \frac{|\xi\circ\varphi|^2}{f\circ\varphi} v^g<\infty$, when $R\rightarrow\infty$, we obtain:
\begin{equation}\label{eq3.15}
     \int_M(f\circ\varphi)|d\varphi|^2 v^g=0.
\end{equation}
Consequently, $|d\varphi|=0$, that is  $\varphi$ is constant.
\end{proof}
From Theorem  \ref{theorem3}, we get the following:
\begin{corollary}
Let $(M,g)$ be a complete non-compact Riemannian manifold and let $\xi$  a conformal vector field on $(M,g)$ with potential function $f>0$ at any point. Then: $$\int_M \frac{|\xi|^2}{f} v^g=\infty.$$
\end{corollary}

\subsection{Bi-harmonic maps and conformal vector fields}
\begin{theorem}\label{theorem4}
Let $(M,g)$ be a compact orientable  Riemannian manifold without boundary,
and let $\xi$  a conformal vector field with non-constant potential function $f$ on a Riemannian manifold $(N,h)$ such
that $\operatorname{grad}^N f $ is parallel. Then,
any bi-harmonic map $\varphi$ from $(M,g)$ to $(N,h)$ is constant.
\end{theorem}
For the proof of Theorem \ref{theorem4}, we need the following lemma.

\begin{lemma}\cite{cherif}\label{lemma4}
Let $(M,g)$ be a compact orientable  Riemannian manifold without boundary
and $(N,h)$  a Riemannian manifold admitting a proper homothetic vector field $\zeta$, i.e. $\mathcal{L}_{\zeta}h=2kh$ with $k\in\mathbb{R}^*$. Then, any bi-harmonic map $\varphi$ from $(M,g)$ to $(N,h)$ is constant.
\end{lemma}

\begin{proof}[Proof of Theorem \ref{theorem4}]
We set $\zeta=[\textrm{grad}^N\,f , \xi]$, since $\operatorname{grad}^Nf$ is parallel on $(N,h)$, then $\zeta$ is an homothetic vector field satisfying $\nabla^N _U \zeta=|\textrm{grad}^N\,f|^2 U$ for any $U\in\Gamma(TN)$ (see \cite{WH}). The Theorem \ref{theorem4} follows from Lemma \ref{lemma4}.
\end{proof}

From Theorem \ref{theorem4}, we deduce:

\begin{corollary}
Let $(M,g)$ be a compact orientable  Riemannian manifold without boundary,
and let $\xi$  a conformal vector field with non-constant potential function $f$ on $(M,g)$. Then,
$\operatorname{grad} f $ is not parallel.
\end{corollary}

\subsection{ Harmonic Maps to Ricci Solitons}
\begin{proposition}\label{proposition6.1}
Let $(M,g)$ be a compact orientable  Riemannian manifold without boundary,
and $(N,h,\xi,\lambda)$  a nontrivial Ricci soliton with:
 \begin{center}
 $\operatorname{Ric}^N>\lambda h$\quad or \quad$\operatorname{Ric}^N<\lambda h$.
 \end{center}
Then any harmonic map $\varphi$ from $(M,g)$ to $(N,h)$ is constant.
\end{proposition}
\begin{proof}
Let $X\in\Gamma(TM)$, we set:
\begin{equation}\label{eq6.1}
    \omega(X)=h\big(\xi\circ\varphi,d\varphi(X)\big),
\end{equation}
let $\{e_i\}$ be a normal orthonormal frame at $x\in M$, we have:
\begin{equation}\label{eq6.2}
    \operatorname{div}^M\omega
   = e_i\big[h\big(\xi\circ\varphi,d\varphi(e_i)\big)\big],
\end{equation}
by equation (\ref{eq6.2}), and the harmonicity condition of $\varphi$, we get:
\begin{eqnarray}\label{eq6.3}
\operatorname{div}^M\omega
   &=& h\big(\nabla^\varphi _{e_i} (\xi\circ\varphi),d\varphi(e_i)\big)
   = \frac{1}{2}(\mathcal{L}_\xi h)(d\varphi(e_i) ,d\varphi(e_i)),
\end{eqnarray}
from the soliton equation,
we find that:
\begin{equation}\label{eq6.4}
  \operatorname{div}^M\omega
  =\lambda  h\big(d\varphi(e_i) ,d\varphi(e_i)\big)-\operatorname{Ric}^N(d\varphi(e_i) ,d\varphi(e_i))
\end{equation}
the Proposition \ref{proposition6.1} follows from equation (\ref{eq6.4}), and the divergence theorem.
\end{proof}

\begin{remark}\label{remark3}
 The condition $\operatorname{Ric}^N>\lambda h$ (resp. $\operatorname{Ric}^N<\lambda h$) is equivalent to
$\operatorname{Ric}^N(v,v)>\lambda h(v,v)$ (resp. $\operatorname{Ric}^N(v,v)<\lambda h(v,v)$),
for any $v\in T_pN-\{0\}$, where $p\in N$.
\end{remark}

It is known that the cigar soliton:  $$(\mathbb{R}^2,\frac{dx^2+dy^2}{1+x^2+y^2}),$$ is steady with strictly positive Ricci tensor (see \cite{H1}),
according to Proposition \ref{proposition6.1}, we have the following:
\begin{corollary}
Any harmonic map $\varphi$ from a compact orientable  Riemannian manifold without boundary to the cigar soliton
is constant.
\end{corollary}
In the case of non-compact Riemannian manifold, we obtain the following results:

\begin{theorem}\label{theorem6}
Let $(M,g)$ be a complete non-compact Riemannian manifold, and $(N,h,\xi,\lambda)$  a nontrivial  Ricci soliton with
$\operatorname{Ric}^N<\mu h$,
for some constant $\mu<\lambda$. If $\varphi:(M,g)\longrightarrow(N,h)$ is harmonic
map, satisfying:
\begin{equation}\label{eq6.5}
    \int_M |\xi\circ\varphi|^2 v^g<\infty,
\end{equation}
then $\varphi$ is constant.
\end{theorem}

\begin{proof}
Let $\rho$ be a smooth function with compact support on $M$, we set:
\begin{equation}\label{eq6.6}
    \omega(X)=h\big(\xi\circ\varphi,\rho^2d\varphi(X)\big),\quad X\in\Gamma(TM).
\end{equation}
Let $\{e_i\}$ be a normal orthonormal frame at $x\in M$, we have:
\begin{equation}\label{eq6.7}
    \operatorname{div}^M\omega
   = e_i\big[h\big(\xi\circ\varphi,\rho^2d\varphi(e_i)\big)\big],
\end{equation}
by equation (\ref{eq6.7}), and the harmonicity condition of $\varphi$, we get:
\begin{eqnarray}\label{eq6.8}
\operatorname{div}^M\omega
   &=& \nonumber h\big(\nabla^\varphi _{e_i} (\xi\circ\varphi),\rho^2d\varphi(e_i)\big)
    +h\big( \xi\circ\varphi,\nabla^\varphi _{e_i}\rho^2d\varphi(e_i)\big)\\
    &=&\rho^2h\big(\nabla^\varphi _{e_i} (\xi\circ\varphi),d\varphi(e_i)\big)
    +2\rho e_i(\rho)h\big(\xi\circ\varphi,d\varphi(e_i)\big),
\end{eqnarray}
by the soliton equation, we find that:
\begin{eqnarray}\label{eq6.9}
  \rho^2h\big(\nabla^\varphi _{e_i} (\xi\circ\varphi),d\varphi(e_i)\big)
  &=&\nonumber\lambda\rho^2  h\big(d\varphi(e_i) ,d\varphi(e_i)\big)\\
  &&-\rho^2 \operatorname{Ric}^N\big(d\varphi(e_i) ,d\varphi(e_i)\big),
\end{eqnarray}
by  Young's inequality we have:
\begin{equation}\label{eq6.10}
    -2\rho e_i(\rho)h\big(\xi\circ\varphi,d\varphi(e_i)\big)
    \leq \epsilon \rho^2|d\varphi|^2+\frac{1}{\epsilon}e_i(\rho)^2|\xi\circ\varphi|^2,
\end{equation}
for all $\epsilon>0$. From (\ref{eq6.8}), (\ref{eq6.9}) and (\ref{eq6.10}) we deduce the inequality:
\begin{eqnarray}\label{eq6.11}
    \lambda\rho^2|d\varphi|^2-\rho^2 \operatorname{Ric}^N\big(d\varphi(e_i) ,d\varphi(e_i)\big)
&-&\nonumber\operatorname{div}^M\omega\\
    &\leq& \epsilon \rho^2|d\varphi|^2+\frac{1}{\epsilon}e_i(\rho)^2|\xi\circ\varphi|^2,
\end{eqnarray}
let $\epsilon=\lambda-\mu$, by (\ref{eq6.11}) we have:
\begin{eqnarray}\label{eq6.12}
   \rho^2\big[ \mu|d\varphi|^2- \operatorname{Ric}^N\big(d\varphi(e_i) ,d\varphi(e_i)\big)\big]
&-&\nonumber\operatorname{div}^M\omega\\
   & \leq& \frac{1}{\lambda-\mu}e_i(\rho)^2|\xi\circ\varphi|^2,
\end{eqnarray}
by the divergence theorem, and (\ref{eq6.12}) we have:
\begin{eqnarray}\label{eq6.13}
      \int_M\rho^2\big[ \mu|d\varphi|^2&-&\nonumber \operatorname{Ric}^N\big(d\varphi(e_i) ,d\varphi(e_i)\big)\big]v^g\\
    &\leq& \frac{1}{\lambda-\mu}\int_M e_i(\rho)^2|\xi\circ\varphi|^2 v^g.
\end{eqnarray}
Consider the smooth function $\rho=\rho_R$ such that, $\rho\leq1$ on $M$, $\rho=1$ on the ball $B(p,R)$, $\rho=0$ on $M\backslash B(p,2R)$
and $|\operatorname{grad}^M \rho|\leq\frac{2}{R}$ (see \cite{pa}). From (\ref{eq6.13}) we get:
\begin{eqnarray}\label{eq6.14}
    \int_M\rho^2\big[ \mu|d\varphi|^2&-&\nonumber \operatorname{Ric}^N\big(d\varphi(e_i) ,d\varphi(e_i)\big)\big]v^g\\
    &\leq& \frac{4}{(\lambda-\mu) R^2}\int_M|\xi\circ\varphi|^2 v^g,
\end{eqnarray}
since $ \int_M |\xi\circ\varphi|^2 v^g<\infty$, when $R\rightarrow\infty$, we obtain:
\begin{equation}\label{eq6.15}
     \int_M\big[ \mu |d\varphi|^2- \operatorname{Ric}^N\big(d\varphi(e_i) ,d\varphi(e_i)\big)\big]v^g=0.
\end{equation}
Consequently, $d\varphi(e_i)=0$, for all $i$ (because $\mu h-\operatorname{Ric}^N>0$),  that is  $\varphi$ is constant.
\end{proof}
If $M=N$ and $\varphi=Id_M$,
from Theorem \ref{theorem6}, we deduce:
\begin{corollary}
Let $(M,g,\xi,\lambda)$ be a complete non-compact nontrivial  Ricci soliton with
$\operatorname{Ric}<\mu h$
 for some constant $\mu<\lambda$.  Then: $$\int_M |\xi|^2 v^g=\infty.$$
\end{corollary}

\subsection{ Bi-harmonic Maps to Ricci Solitons}

\begin{theorem}\label{theorem7}
Let $(M,g)$ be a compact orientable  Riemannian manifold without boundary,
and $(N,h,\xi,\lambda)$  a nontrivial Ricci soliton with:
\begin{center}
$\operatorname{Ric}^N>\lambda h$ \quad or\quad $\operatorname{Ric}^N<\lambda h$.
\end{center}
Suppose that $\xi$ is Jacobi-type vector field.
Then any bi-harmonic map $\varphi$ from $(M,g)$ to $(N,h)$ is constant.
\end{theorem}

\begin{proof}
We set:
\begin{equation}\label{eq7.7}
    \eta(X)=h\big(\xi\circ\varphi,\nabla^\varphi _X \tau(\varphi)\big),\quad X\in\Gamma(TM),
\end{equation}
calculating in a normal frame at $x \in M$, we have:
\begin{eqnarray}\label{eq7.8}
 \operatorname{div}^M \eta
   &=&\nonumber e_i\big[h\big(\xi\circ\varphi,\nabla^\varphi _{e_i} \tau(\varphi)\big)\big]  \\
   &=&  h\big(\nabla^\varphi _{e_i}(\xi\circ\varphi),\nabla^\varphi _{e_i} \tau(\varphi)\big)
        + h\big(\xi\circ\varphi,\nabla^\varphi _{e_i}\nabla^\varphi _{e_i} \tau(\varphi)\big),
\end{eqnarray}
from equation (\ref{eq7.8}), and the bi-harmonicity condition of $\varphi$, we get:
\begin{eqnarray}\label{eq7.9}
 \operatorname{div}^M \eta
   &=&  \nonumber h\big(\nabla^\varphi _{e_i}(\xi\circ\varphi),\nabla^\varphi _{e_i} \tau(\varphi)\big)\\
        &&-h\big(R^N(\tau(\varphi),d\varphi(e_i))d\varphi(e_i),\xi\circ\varphi\big),
\end{eqnarray}
the first term on the left-hand side of (\ref{eq7.9}) is
\begin{eqnarray}\label{eq7.10}
h\big(\nabla^\varphi _{e_i}(\xi\circ\varphi),\nabla^\varphi _{e_i} \tau(\varphi)\big)
   &=&\nonumber e_i\big[h\big(\nabla^\varphi _{e_i}(\xi\circ\varphi),\tau(\varphi)\big) \big]\\
       &&-h\big(\nabla^\varphi _{e_i} \nabla^\varphi _{e_i}(\xi\circ\varphi),\tau(\varphi)\big),
\end{eqnarray}
by equations (\ref{eq7.9}), (\ref{eq7.10}), and the following property:
$$h(R^N(X,Y)Z,W)=h(R^N(W,Z)Y,X),$$
where $X,Y,Z,W\in\Gamma(TM)$, we conclude that:
\begin{eqnarray}\label{eq7.11}
 \operatorname{div}^M \eta
   &=&\nonumber \operatorname{div}^M h\big(\nabla^\varphi _{\cdot}(\xi\circ\varphi),\tau(\varphi)\big)
   -h\big(\nabla^\varphi _{e_i} \nabla^\varphi _{e_i}(\xi\circ\varphi),\tau(\varphi)\big)\\
   && -h\big(R^N(\xi\circ\varphi,d\varphi(e_i))d\varphi(e_i),\tau(\varphi)\big),
\end{eqnarray}
since $\xi$ is a Jacobi-type vector field, we have:
\begin{eqnarray}\label{eq7.12}
 \operatorname{div}^M \eta
   &=& \operatorname{div}^M h\big(\nabla^\varphi _{\cdot}(\xi\circ\varphi),\tau(\varphi)\big)-h\big(\nabla^N _{\tau(\varphi)} \xi,\tau(\varphi)\big),
\end{eqnarray}
by the soliton equation, we get:
\begin{eqnarray}\label{eq7.13}
 \operatorname{div}^M \eta
   &=&\nonumber \operatorname{div}^M h\big(\nabla^\varphi _{\cdot}(\xi\circ\varphi),\tau(\varphi)\big)\\&&-\lambda|\tau(\varphi)|^2+\operatorname{Ric}^N(\tau(\varphi),\tau(\varphi)),
\end{eqnarray}
from equation (\ref{eq7.13}), and the divergence theorem, with $\operatorname{Ric}^N<\lambda h$ (or $\operatorname{Ric}^N>\lambda h$), we get $\tau(\varphi)=0$, i.e. $\varphi$ is harmonic map, so by the Proposition \ref{proposition6.1},  $\varphi$ is constant.
\end{proof}
From Theorem \ref{theorem7}, we deduce:

\begin{corollary}
Let $(M,g,\xi,\lambda)$ be a compact nontrivial Ricci soliton with:
\begin{center}
$\operatorname{Ric}>\lambda g$ \quad or \quad$\operatorname{Ric}<\lambda g$.
\end{center}
Then $\xi$ is not Jacobi-type vector field.
\end{corollary}


\begin{thebibliography}{99}


\bibitem{BW} P. Baird, J. C. Wood,  {\it Harmonic morphisms between Riemannain manifolds},
Clarendon Press Oxford, 2003.

\bibitem{BL1} D. E. Blair, {\it Contact Manifolds in Riemannian Geometry}, Lecture Nots in Mathematics 509,
              Springer, 1976, pp 17-35.

\bibitem{CMO} R. Caddeo, S. Montaldo, C. Oniciuc, {\it Biharmonic submanifolds of $S^{3}$}. { Int. J. Math.}, \textbf{ 12} (2001), 867-876.
\bibitem{C} N. Course, {\it $f$-harmonic maps which map the boundary of the domain to
one point in the target}, { New York Journal of Mathematics}, \textbf{ 13} (2007), 423-435.


\bibitem{Deshmukh} S. Deshmukh, {\it Jacobi-type vector fields and Ricci soliton}, { Bull. Math. Soc.
Sci. Math. Roumanie}, \textbf{ 55(103)(1)} (2012), 41-50.
\bibitem{Duggal} K. L. Duggal, R. Sharma, {\it Symmetries of Spacetimes
and Riemannian Manifolds}, Kluwer Academic Publishers, Dordrecht, 1999.

\bibitem{ES} J. Eells and J. H. Sampson, {\it Harmonic mappings of Riemannian manifolds}, { Amer. J. Math.}
\textbf{ 86} (1964), 109-160.

\bibitem{H1} R. Hamilton, {\it The Ricci flow on surfaces}, Contemp. Math \textbf{71} (1988) 237-262.
\bibitem{H2} R. Hamilton, {\it The Harnack estimate for the Ricci flow}, J. D. G. \textbf{37} (1993) 225-243.


\bibitem{Jiang} G. Y. Jiang, {\it 2-Harmonic maps between Riemannian manifolds}, Annals of
Math., China, \textbf{ 7A(4)}(1986), 389-402.


\bibitem{WH} W. Kühnel, H. Rademacher, { \it Conformal vector fields on pseudo-Riemannian
spaces}, { Differential Geometry and its Applications} \textbf{ 7} (1997), 237-250.
\bibitem{LO} E. Loubeau and C. Oniciuc, {\it On the biharmonic and harmonic indices of the hopf map}, { Transactions of the
american mathematical society,} \textbf{ 359 (11)} (2007), 5239-5256.
\bibitem{cherif} A. Mohammed Cherif, {\it Some results on harmonic and bi-harmonic maps}, International Journal of Geometric Methods in Modern Physics,  \textbf{14} (2017).

\bibitem{ON} O'Neil, {\it Semi- Riemannian Geometry}, Academic Press, New York, 1983.
\bibitem{OND} S. Ouakkas, R. Nasri and M. Djaa, {\it On the $f$-harmonic and $f$-biharmonic maps}, { J. P. Journal.
of Geom. and Top.,}  \textbf{  10 (1)} (2010), 11-27.

\bibitem{SMA} S. Pigola, M. Rimoldi and A. G. Setti, {\it Remarks on non-compact gradient Ricci
solitons}, { Math. Z.} \textbf{ 268} (2011), 777-790.

\bibitem{YX} Y. Xin, {\it Geometry of harmonic maps}, Fudan University, 1996.
\bibitem{yano} K. Yano and T. Nagano, {\it The de Rham decomposition, isometries and affine transformations
in Riemannian space}, { Japan. J. Math.}, \textbf{ 29} (1959), 173-184.
\bibitem{pa} S. T. Yau, {\it Harmonic functions on complete Riemannian manifolds},{ Comm. Pure Appl.}
Math. \textbf{ 28} (1975),  201-228.




\end{thebibliography}
\end{document}